%% file: root.tex
\documentclass[letterpaper, 10 pt, conference]{ieeeconf}  

\IEEEoverridecommandlockouts                              

\usepackage{comment}
\usepackage{graphicx}
\usepackage{dsfont}
\usepackage{amsmath,amssymb}
\newtheorem{definition}{Definition}
\newtheorem{remark}{Remark}
\newtheorem{lemma}{Lemma}
\newtheorem{assumption}{Assumption}
\newtheorem{theorem}{Theorem}
\newtheorem{algorithm}{Algorithm}
\newtheorem{proposition}{Proposition}
\newtheorem{problem}{Problem}
\newtheorem{corollary}{Corollary}
\usepackage{subcaption}
\usepackage{mathtools}

\usepackage{color,soul}

\title{\LARGE \bf
Learning in Memristive Electrical Circuits
}

\author{H. M. Heidema, H. J. van Waarde, and B. Besselink
\thanks{The authors are with the Jan C. Willems Center for Systems
and Control, and the Bernoulli Institute for Mathematics, Computer
Science, and Artificial Intelligence, University of Groningen, The Netherlands. 
Marieke Heidema and Bart Besselink are also with CogniGron (Groningen Cognitive Systems and Materials Center), University of Groningen, The Netherlands. Marieke Heidema and Bart Besselink acknowledge the financial support of the CogniGron research center and the Ubbo Emmius Funds. 
Henk van Waarde acknowledges financial support by the Dutch Research Council under the NWO Talent Programme Veni Agreement (VI.Veni.222.335). Email: {\tt\small h.m.heidema@rug.nl; 
h.j.van.waarde@rug.nl; b.besselink@rug.nl.}}
}

\begin{document}
\mathtoolsset{showonlyrefs}

\maketitle
\thispagestyle{empty}
\pagestyle{empty}

\begin{abstract}
\input{Abstract/abstract}

\end{abstract}

\section{INTRODUCTION}
\input{Introduction/introduction}

\section{MEMRISTOR NETWORKS}
\input{Memristors/singleMemristor}

\input{Memristors/crossbarArray}

\input{Memristors/problemStatement}

\section{READING}
\input{Reading/reading}

\section{WRITING}
\input{Writing/switches}

\input{Writing/continuous}

\section{APPLICATIONS}
\input{Applications/MatrixProduct}

\input{Applications/LeastSquares}

\section{CONCLUSION}
\input{Conclusion/conclusion}

\addtolength{\textheight}{-12cm}

\bibliographystyle{IEEEtran}
\bibliography{References}

\end{document}

%% file: Abstract/abstract.tex
Memristors are nonlinear two-terminal circuit elements whose resistance at a given time depends on past electrical stimuli. 
Recently, networks of memristors have received attention in neuromorphic computing since they can be used as a tool to perform linear algebraic operations, like matrix-vector multiplication, directly in hardware.
In this paper, the aim is to resolve two fundamental questions pertaining to a specific, but relevant, class of memristive circuits called crossbar arrays.
In particular, we show (1) how the resistance values of the memristors at a given time can be determined from external (voltage and current) measurements, and (2) how the resistances can be steered to desired values by applying suitable external voltages to the network. 
The results will be applied to solve a prototypical learning problem, namely linear least squares, by applying and measuring voltages and currents in a suitable memristive circuit.

%% file: Introduction/introduction.tex
As computing technology has gotten more complex over the past decades, one of the main aims has always been to lower its energy consumption. 
In the search for more energy-efficient computers, the potential of neuromorphic computing has been studied \cite{Ribar2021}, \cite{Indiveri2015}. 
This is a method of designing analog information processing systems by modeling them after the nervous system in the brain. 
Such computing technology could lead to a significant reduction in energy consumption, when compared to existing digital computers, see \cite{Mead1990}. 

To this end, there is a need for devices that behave like synapses in order to emulate biological neural networks \cite{Thomas2013}, \cite{Snider2011},  \cite{Khalid2019}, and \cite{Sah2014}. 
Memristors are suitable candidates for synapses, as they have nonlinear dynamics with memory.   
In this study, we will therefore analyze electrical circuits with memristors. Memristors were introduced by Chua \cite{Chua1971} as the fourth electrical circuit element. 
They can be regarded as resistors for which the instantaneous resistance depends on past external stimuli  and that, in  the absence of external stimuli, retain their resistance value.  When regarding this resistance as the memory of a memristor, a change in resistance can be seen as learning. Since this mimics the learning processes in biological synapses, memristive materials can act as synapses in neuromorphic computing systems. 

Past research on memristors introduced the notions of charge- and flux-controlled memristors, and discussed their passivity properties \cite{Corinto2016}, \cite{Corinto2021}. Furthermore, past research investigates monotonicity of (memristive) circuits \cite{Chaffey2024} and monotonicity in relation to passivity \cite{Corinto2015}. 
These studies present a way of modeling electrical circuits with memristors, using graph theory and Kirchhoff's laws \cite{Huijzer2023}. 
Here, we will make use of this to analyze memristive circuits. 

In particular, we will be studying memristive crossbar arrays. These are networks consisting of row and column bars, with memristors on the cross-points. 
Crossbar arrays have been studied in the case of resistors \cite{Sun2019} and memristors \cite{Sebastian2020}. Here, it has been found that this network structure can be used to compute matrix-vector products, where the (instantaneous) resistance values of the elements correspond to the entries of the matrix.
In the case of a resistive crossbar array, this means that the matrix we can perform matrix-vector products with is fixed. However, in the case of memristive crossbar arrays, the memristors can change their resistance value based on external stimuli, giving a set of matrices with which to compute matrix-vector products. 

What is missing from the past research, is a way to steer the resistance values of the memristors  so that they coincide with the entries of a given matrix. 
To this end, the contributions of this paper are as follows: (1) we define memristors and introduce memristive crossbar arrays; (2) we show how one can determine (\textit{read}) the resistance values of the memristors at a given time in such a circuit, based on voltage and current measurements; (3) we show how one can steer (\textit{write}) the resistance values of the memristors in such a circuit to desired values, by application of external stimuli to the network. Together, these results enable us to compute matrix-vector products and to solve least squares problems using memristive crossbar arrays. 

The remainder of this paper is organized as follows. In Section II, we introduce memristors and memristive crossbar arrays. Section III will then concern reading the instantaneous resistance (conductance) values of the memristors. Section IV deals with writing the resistance (conductance) values of the memristors. Section V will concern applications of our theory to matrix-vector multiplication and solving least-squares problems. This paper then ends with  a conclusion and discussion in Section VI.

\subsubsection*{Notation} We denote the element in row $k$ and column $l$ of $A\in\mathds{R}^{m\times n}$ as $A_{kl}$ or $[A]_{kl}$. The Kronecker product of two matrices $A\in\mathds{R}^{m\times n}$ and $B\in\mathds{R}^{p\times q}$ is denoted by $A\otimes B \in\mathds{R}^{mp\times nq}$. 
The vectorization of a matrix $C\in\mathds{R}^{m\times n}$, denoted by $\text{vec}(C)$, is the vector $\text{vec}(C)=\begin{bmatrix} c_1^\top & c_2^\top & \hdots & c_n^\top \end{bmatrix}^\top\in\mathds{R}^{mn}$, where $c_l\in\mathds{R}^m$ is the $l$-th column of $C$. 
Let $D_1, D_2,\ldots,D_r$ be real matrices. The matrix with block-diagonal entries $D_1, D_2,\ldots,D_r$ is denoted by $\text{diag}(D_1,D_2,\hdots, D_r)$.
The column vector of size $m$ with all ones is denoted by $\mathds{1}_m$. 
The image of a linear map $F:\mathds{R}^n\rightarrow \mathds{R}^m$ is denoted by $\text{im}(F)$.
 The set of all matrices $M\in\mathds{R}^{m\times n}$ with positive entries is denoted by $\mathds{R}_{>0}^{m\times n}$.

%% file: Memristors/singleMemristor.tex
\subsection{Memristors}
In this paper, we consider networks of memristors. A memristor $M_{kl}$ is a two-terminal electrical circuit element, originally postulated by Chua \cite{Chua1971}. It provides a relation between the (magnetic) flux $\varphi_{kl}$ and (electric) charge $q_{kl}$, which satisfy \vspace{-0.2cm}
\begin{equation}
    \tfrac{d}{dt} \varphi_{kl} = V_{kl}, \quad \tfrac{d}{dt} q_{kl} = I_{kl},
\end{equation}
with $V_{kl}$ the voltage across and $I_{kl}$ the current through the memristor. In particular, we consider so-called flux-controlled memristors of the form
\begin{align}
     q_{kl} &= g_{kl}(\varphi_{kl}), \label{eqn:flux controlled}
\end{align} 
for some function $g_{kl}:\mathds{R}\rightarrow\mathds{R}$. 
\begin{assumption}\label{assumption 1}
    The function $g_{kl}(\cdot)$ is continuously differentiable, and strictly monotone, i.e.,
    $$(g_{kl}(\varphi)-g_{kl}(\varphi')) (\varphi-\varphi')>0$$
for all $\varphi,\varphi'\in\mathds{R}$ such that $\varphi\neq \varphi'$. 
\end{assumption}

Taking the derivative of \eqref{eqn:flux controlled} with respect to time, we obtain the following dynamical system describing the memristor
\begin{equation}
\begin{aligned}
    \tfrac{d}{dt} \varphi_{kl}(t) &= V_{kl}(t),  \quad 
   I_{kl}(t) = W_{kl}(\varphi_{kl}(t)) V_{kl}(t), \label{eqn:flux controlled v2}
\end{aligned}
\end{equation} 
where $W_{kl}(\varphi_{kl}) =  \tfrac{dg_{kl}(\varphi_{kl})}{d\varphi_{kl}}$. Here, we will refer to $W_{kl}(\varphi_{kl}(t))$ as the \textit{memductance} of the memristor at time $t$. Note that, by Assumption \ref{assumption 1}, the memductance $W_{kl}(\varphi)$ is positive for all $\varphi\in\mathds{R}$. 
\begin{assumption}\label{assumption 2}
    The function $W_{kl}(\cdot)$ is strictly monotone and $\beta-$Lipschitz continuous, where the latter means that
    \begin{equation}
        |W_{kl}(\varphi)-W_{kl}(\varphi')| \leq \beta |\varphi-\varphi'|
    \end{equation}
for some $\beta>0$ and all $\varphi,\varphi'\in\mathds{R}$. 
\end{assumption}

    Note that \eqref{eqn:flux controlled v2} is similar to the description of a linear resistor, only the conductance is not constant but depends on past external stimuli such as the voltage over the memristor.     
    {The changes in memductance due to external stimuli are what we view as a learning process.}

\begin{remark}
    Assumption \ref{assumption 1} implies passivity of the memristor in \eqref{eqn:flux controlled v2}, see Theorem 6 in \cite{Chua1971}. Furthermore, the HP memristor \cite{Strukov2008}, a frequently used example of a physical memristor, satisfies both Assumptions \ref{assumption 1} and \ref{assumption 2}.  
\end{remark}

%% file: Memristors/crossbarArray.tex
\subsection{Memristive crossbar arrays} 
We now consider a network of memristors $M_{kl}$ with the structure depicted in Figure \ref{fig:memristive crossbar array}. In particular, the memristors are arranged as an array with $m$ rows and $n$ columns, where we denote the memristor connecting row $k$ with column $l$ by $M_{kl}$ and its flux by $\varphi_{kl}$. 
 Furthermore, with each memristor $M_{kl}$, we associate a switch (also called selector) $S_{kl}$, for $k\in\{1,\ldots, m\}$ and $l\in\{1,\ldots, n\}$. 
\begin{remark}
    In Section IV.A, we will discuss the need to include switches in the memristive crossbar array, contrary to resistive crossbar arrays studied in \cite{Sun2019}.
\end{remark}

\begin{figure}[ht]
    \centering
    \includegraphics[width=0.38\textwidth]{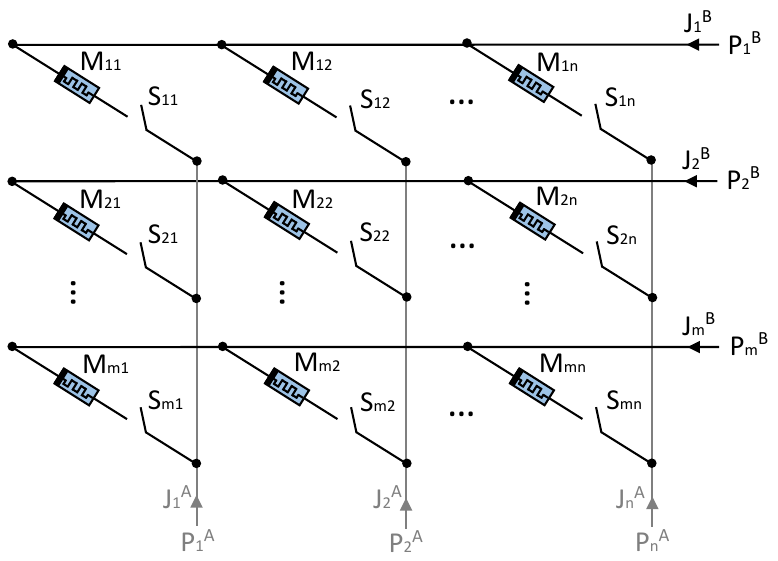}
    \caption{Memristive crossbar array with switches.}
    \label{fig:memristive crossbar array}
\end{figure}

In order to model these memristive crossbar arrays, we collect the voltages $V_{kl}$ and currents $I_{kl}$ associated with the memristors as
\begin{align}
    V &= \begin{bmatrix} V_{11} & V_{21} &\cdots & V_{m1}  & V_{12} & \cdots &  V_{mn} \end{bmatrix}^\top, \label{eqn: V} \\
    I &= \begin{bmatrix} I_{11} & I_{21} &\cdots & I_{m1}  & I_{12} & \cdots &  I_{mn} \end{bmatrix}^\top. \label{eqn: I}
\end{align}
Here, we have omitted the argument $t$ for simplicity. 
Furthermore, $P(t)\in \mathds{R}^{m+n}$ and $J(t)\in \mathds{R}^{m+n}$ respectively  denote the vector of voltage potentials at and currents through the (external) terminals. Again omitting the argument $t$, they are given by
\begin{align}
    P &= \begin{bmatrix} P_1^A  &  \cdots & P_n^A & P_1^B  &  \cdots & P_m^B\end{bmatrix}^\top, \\ 
    J &= \begin{bmatrix} J_1^A  &  \cdots & J_n^A & J_1^B &  \cdots & J_m^B\end{bmatrix}^\top.
\end{align}

Using the Kronecker product, the \textit{incidence matrix} corresponding to this network of memristors with all switches being closed, can easily be found to be given by
   \begin{equation}\label{eqn: incidence matrix}
       D = \begin{bmatrix}
       I_n\otimes \mathds{1}_m^\top \\
       -\mathds{1}_n^\top \otimes I_m
   \end{bmatrix} \in \mathds{R}^{(n+m)\times mn}. 
   \end{equation}  
In case not all switches are closed, the incidence matrix can be written of the form $\Bar{D}=DS(t),$
where $S(t)$ is an $mn$ by $mn$ diagonal matrix, which has ones and zeroes on its diagonal, depending on the switches $S_{kl}$ being open or closed at time $t$. In particular, the matrix $S$ is given by 
\begin{equation}
    S=\text{diag} \begin{pmatrix}
    s_{11},\ s_{21}, \ \ldots, \ s_{m1}, \ s_{12}, \ \ldots, \ s_{mn}
\end{pmatrix},
\end{equation} 
where $s_{kl}\in\{0,1\}$ is defined as
\begin{equation}
    s_{kl} = \left\{
	\begin{array}{ll}
		1  & \mbox{if switch } S_{kl} \text{ is closed}, \\
		0 & \mbox{if switch } S_{kl} \text{ is open.}
	\end{array}
\right.
\end{equation}

Now, Kirchhoff's current and voltage law say that
\begin{equation}
    J=DSI \quad \text{and} \quad V=S D^{\top} P. \label{eqn: 3 laws}
\end{equation} 
Denote the vector of fluxes over the memristors by
\begin{equation}
    \varphi = \begin{bmatrix} \varphi_{11} & \varphi_{21} & \cdots & \varphi_{m1}  & \varphi_{12} &  \cdots & 
 \varphi_{mn} \end{bmatrix}^\top.
\end{equation} 
In the remainder of this paper, we let $\varphi_{kl}$ refer to its $(k+m(l-1))$-th element. Now, write 
\begin{align}
    W(\varphi)=\text{diag} &\left(W_{11}(\varphi_{11}), \  W_{21}(\varphi_{21}), \ \cdots, W_{m1}(\varphi_{m1}), \right. \\
    & \ \left. W_{12}(\varphi_{12}), \ \cdots, \ W_{mn}(\varphi_{mn}) \right),
\end{align}
then the dynamics of the memristors in the array is given by
\begin{equation}
\begin{aligned}
    \tfrac{d}{dt} \varphi &= V, \quad    I=W(\varphi)V. \label{eqn:general crossbar array v2}
\end{aligned}
\end{equation}
Combining \eqref{eqn: 3 laws} and \eqref{eqn:general crossbar array v2}, we find that the dynamics of the memristive crossbar array with switches is given by
\begin{equation}
\begin{aligned}
    \tfrac{d}{dt} \varphi &= S D^\top P, \quad J=DSW(\varphi)S D^\top P. \label{eqn:general crossbar array}
\end{aligned}
\end{equation}

\begin{remark}
    From the dynamics of the memristive crossbar array, note that \textit{any} $\Bar{\varphi}\in\mathds{R}^{mn}$ is an equilibrium point of \eqref{eqn:general crossbar array} for $P=0$. In other words, when there are no external stimuli, the value $W_{kl}(\Bar{\varphi}_{kl})$ is maintained. 
    This means that the memristive array can be used to store information.
\end{remark}

%% file: Memristors/problemStatement.tex
\subsection{Problem statement} 
Viewing the memductance values $W_{kl}(\varphi_{kl})$ as stored information, we would like to be able to retrieve those values on the basis of an experiment on the terminals, i.e., using functions $P$ and $J$. This leads to the following definition.
\begin{definition}
    Given $S\in \{0,1\}^{mn\times mn}$ and $P,J:[0,T]\rightarrow \mathds{R}^{n+m}$, the \textit{set of consistent memductance matrices} is given by \vspace{-0.1cm}
    \begin{align}
       &\mathcal{W}(P,J):= \left\{ W\in\mathds{R}^{mn\times mn} \ | \ \exists \varphi :[0,T]\rightarrow \mathds{R}^{mn}  \right. \\
        &\  \text{ such that } (P,\varphi,J) \text{ satisfy } \text{\eqref{eqn:general crossbar array} with } \left. W( \varphi(0))=W   \right\}. 
    \end{align} 
\end{definition} \vspace{0.1cm}

Hence, the set of consistent memductance matrices consists of all matrices with initial memductance values that agree with the given terminal behavior $P$ and $J$. 

In the remainder of this paper, we assume that the voltage potentials can be chosen directly through (controlled) voltage sources. Hereto, we note that \eqref{eqn:general crossbar array} can be regarded as a dynamical system with state trajectory $\varphi$, input $P$ and output $J$. Here, $\varphi$ depends solely on $P$ and the initial condition $\varphi(0)=\varphi_0$, which we denote as $\varphi_{\varphi_0,P}$. Furthermore, we denote the corresponding currents by $J_{\varphi_0,P}$ and assume them to be measurable.
 This allows us to state the following: 

\begin{problem}[Reading Problem]\label{reading problem}
    Let $S=I$. Find $T\geq 0$ and $P:[0,T]\rightarrow \mathds{R}^{n+m}$ such that, for any $\varphi_0$,
    \begin{enumerate}
        \item $\mathcal{W}(P,J_{\varphi_0,P})$ is a singleton;
        \item $\varphi_{\varphi_0,P}(T) = \varphi_0$.
    \end{enumerate}
\end{problem} 

So, the reading problem is to choose the terminal behavior $P$ on the time-interval $[0,T]$ in such a way that, together with the corresponding current $J_{\varphi_0,P}$ that we measure on that time-interval, we can uniquely determine the initial memductance matrix for any $\varphi_0$. In addition, we want $P$ to be such that the memductance matrix at the end of our experiment, i.e., $W(\varphi(T))$, is equal to the memductance matrix from before the experiment, i.e., $W(\varphi_0)$.

Apart from \textit{reading} the memductance values at a given time, we also want to be able to steer these to desired values. This is formalized below as the \textit{writing problem}.
\begin{problem}[Writing Problem] \label{writing problem}
    Let $W_d\in\mathds{R}_{>0}^{mn\times mn}$ be a diagonal matrix of desired memductance values and let $\epsilon>0$. Then, for any $\varphi_0\in\mathds{R}^{mn}$, find $T\geq 0$,  $P:[0,T]\rightarrow \mathds{R}^{n+m}$, and $S:[0,T]\rightarrow \{0,1\}^{mn\times mn}$
    such that 
    $$\bigl|W_{d,kl} - [W(\varphi_{\varphi_0,P}(T))]_{kl}\bigr|\leq \epsilon$$
    for all $k\in\{ 1,\hdots,m\}$ and $l\in\{1,\hdots,n\}$.    
\end{problem}

%% file: Reading/reading.tex
\input{Reading/voltageSources}

%% file: Reading/voltageSources.tex
In this section, we assume that all the switches are closed, i.e., $s_{kl}=1$ for all $k\in\{ 1,\hdots,m\}$ and $l\in\{1,\hdots,n\}$.
In order to read the memductance values of the memristors in the memristive crossbar array, let us add $n$ grounded voltage sources at the column terminals and ground the row terminals, i.e., $P^B_k=0$ for $k\in\{ 1,\hdots,m\}$. 
The dynamics of \eqref{eqn:general crossbar array} now leads to \vspace{-0.2cm} 
\begin{equation}
    \tfrac{d}{dt} \varphi_{kl} = P_l^A, \label{eqn: reading dynamics}
\end{equation} 
due to the specific structure of the incidence matrix $D$, see \eqref{eqn: incidence matrix}. Moreover, the second equation in \eqref{eqn:general crossbar array} leads to 
\begin{equation}     
-J_k^B =  \sum_{j=1}^n W_{kj}(\varphi_{kj}) P^A_j. \label{eqn: reading} 
\end{equation}

Now, let $\tau>0$ and choose the input voltage signals as 
\begin{equation}
    P_l^A(t) = \left\{
        \begin{array}{ll}
            0 & \quad t\in [0, t_l-2\tau)   \\
            -1 & \quad t \in [t_l-2\tau, t_l-\tau) \\ 
            1 & \quad t \in [t_l-\tau, t_l+\tau) \\
            -1 & \quad t \in [t_l+\tau, t_l+2\tau) \\
            0 & \quad t\in [t_l+2\tau,t_n+2\tau) 
        \end{array}
    \right. \label{eqn: V reading}
\end{equation}
 for all $l\in\{1,\ldots n\}$. Here, $\{t_l\}_{l=1}^n$ is a sequence satisfying $t_1\geq 2\tau$ and $t_{l+1}\geq t_l+4\tau$ for all $l\in\{1,\ldots n-1\}$.
 
From \eqref{eqn: reading dynamics} it then follows that the corresponding flux of each memristor can be computed as
\begin{equation}
   \varphi_{kl}(t) - \varphi_{kl}(0) = \int_0^t P_l^A(s) ds, \quad \varphi_{kl}(0)=\varphi_{kl,0}. \label{eqn: flux reading}
\end{equation}
By \eqref{eqn: V reading} it is clear that the flux has the following property:
\begin{equation}
    \varphi_{kl}(0) = \varphi_{kl}(t_l) = \varphi_{kl}(T) \label{eqn: flux property}
\end{equation} 
for any $T\geq t_l+2\tau$. Hence, the time horizon $T:= t_n+2\tau$ and the function $P:[0,T]\rightarrow \mathds{R}^{n+m}$, defined by $P_k^B=0$ and \eqref{eqn: V reading}, satisfy item 2) of Problem \ref{reading problem}. 

It turns out that $P$ also satisfies item 1), and is thus a solution to Problem \ref{reading problem}. This is stated in the following:
\begin{theorem} \label{thm: reading}
Let $T:=t_n+2\tau$ and define $P:[0,T]\rightarrow \mathds{R}^{n+m}$ by $P_k^B=0$ and \eqref{eqn: V reading}, for $k\in\{ 1,\hdots,m\}$ and $l\in\{1,\hdots,n\}$. For any $\varphi_0\in\mathds{R}^{mn}$, $\mathcal{W}(P,J_{\varphi_0,P})$ is a singleton and $\varphi_{\varphi_0,P}(T) = \varphi_0$. 
\end{theorem}
\begin{proof}
    We have already shown that $\varphi_{\varphi_0,P}(T) = \varphi_0$. Therefore, we only need to show that $\mathcal{W}(P,J_{\varphi_0,P})$ is a singleton. Take any $W^1, W^2\in \mathcal{W}(P,J_{\varphi_0,P})$. Then, $\exists \varphi^1$ and $\varphi^2$ such that $W(\varphi^1(0))=W^1$, respectively, $W(\varphi^2(0))=W^2$. 
    By definition of our voltage signals and equations \eqref{eqn: reading} and \eqref{eqn: flux property}, it is clear that
\begin{equation}
    -J_k^B(t_l) = W_{kl}(\varphi_{kl}^1(0)) \text{ and } -J_k^B(t_l) = W_{kl}(\varphi_{kl}^2(0)),  
\end{equation}
for all $k\in\{ 1,\hdots,m\}$ and $l\in\{1,\hdots,n\}$. 
Due to strict monotonicity, see Assumption \ref{assumption 2}, this implies that $\varphi_{kl}^1(0) = \varphi_{kl}^2(0)$
for all $k\in\{ 1,\hdots,m\}$ and $l\in\{1,\hdots,n\}$.
Hence, $W^1= W(\varphi^1(0)) = W(\varphi^2(0)) = W^2$. We conclude that $\mathcal{W}(P,J_{\varphi_0,P})$ is a singleton, which proves the theorem.
\end{proof}

Note that, from the proof of Theorem \ref{thm: reading}, it follows that we can \textit{read} the initial memductance value $W_{kl}(\varphi_{kl}(0))$ of each memristor by measuring the current $-J_k^B(t_l)$.

%% file: Writing/switches.tex
\subsection{Why switches?}
In reading the initial memductance values of the memristors we have not made use of the switches. However, as will be explained next, these switches are crucial in steering the memductances to desired values. 
Here, we will assume, contrary to the previous section, that there are voltage sources at both the column and row terminals, potentially allowing for more steering possibilities.
    With this, the array is such that the memductance values of memristors in the same 
    column (row) depend on each other due to same input on that column (row). 
Now, to show that switches are crucial in steering, we will first assume that $S=I$. Based on the first equation of \eqref{eqn:general crossbar array}, this implies that the set of matrices of memductance values that can be reached from $\varphi_0$ by suitable choice of $P$, is given by 
\begin{align}
    \mathcal{R}(\varphi_0)&=\left\{ \left. X \in \mathds{R}_{>0}^{m\times n}  \right\vert \exists \varphi \in\mathds{R}^{mn} \text{ such that } \right. \\
    &\hspace{0.8cm} \varphi-\varphi_0\in \text{im} (D^\top) \text{ and } X_{kl} = W_{kl}(\varphi_{kl}) \\
    &\hspace{0.7cm}   \text{ for all} \left. k\in\{1,\ldots,m\} \text{ and } l\in\{1,\ldots,n\} \right\}.
\end{align}   
    Second of all, let the set of all desired matrices of memductance values be denoted by 
    \begin{align} 
        \mathcal{D}&=\left\{ \left. X \in \mathds{R}_{>0}^{m\times n}  \right\vert \exists \varphi \in \mathds{R}^{mn} \text{ s.t. } X_{kl} = W_{kl}(\varphi_{kl}) \right. \\
        &\hspace{0.6cm}  \text{ for all} \left. k\in\{1,\ldots,m\} \text{ and } l\in\{1,\ldots,n\} \right\}. \label{eqn: desired memductance matrices set}
    \end{align} 
    This includes all matrices one could create based on the limitations of each memristor separately without considering the limitations the crossbar array dynamics \eqref{eqn:general crossbar array} puts on the possible matrix elements that can be created.  
Using these sets, we can formalize the observation that switches are crucial in writing. 

\begin{proposition} \label{proposition 1}
    Consider a crossbar array characterized by incidence matrix $D$ in \eqref{eqn: incidence matrix} with $m\geq 2$ by $n\geq 2$. If $s_{kl}=1$ for all $k\in\{1,\ldots, m\}$ and $l\in\{1,\ldots, n\}$, then for each $\varphi_0\in\mathds{R}^{mn}$, $\mathcal{R}(\varphi_0)\subsetneq \mathcal{D}$. 
\end{proposition}

This proposition thus says that there exist desired memductance matrices that can not be created using the $n+m$ voltage sources at the terminals. Now, before we prove this proposition, we first state and prove the following lemma.
\begin{lemma} \label{lemma 1}
    Consider any $X,Y\in\mathcal{D}$ and let $\varphi,\theta\in\mathds{R}^{mn}$ be such that $X_{kl}=W_{kl}(\varphi_{kl})$ and $Y_{kl}=W_{kl}(\theta_{kl})$. If $X=Y$, then $\varphi=\theta$.
\end{lemma}
\begin{proof}
    If $X=Y$, then we know that \vspace{-0.1cm}
    \begin{equation}
        \sum_{k=1}^{m} \sum_{l=1}^n (W_{kl}(\varphi_{kl})-W_{kl}(\theta_{kl}))(\varphi_{kl}-\theta_{kl}) = 0. \label{eqn: strictly monotone} 
    \end{equation} 
    By Assumption \ref{assumption 2},  \vspace{-0.1cm}
    \begin{equation}
        (W_{kl}(\varphi_{kl})-W_{kl}(\theta_{kl}))(\varphi_{kl}-\theta_{kl}) > 0, \quad \forall \varphi_{kl}\neq \theta_{kl}.
    \end{equation} 
    For the equality in  \eqref{eqn: strictly monotone} to hold, we must hence have that $\varphi_{kl}=\theta_{kl}$ for all $k\in\{1,\ldots, m\}$ and $l\in\{1,\ldots, n\}$. Hence, we conclude that $\varphi=\theta$.
\end{proof}

Using Lemma \ref{lemma 1}, we can now prove the proposition.

\quad \emph{Proof of Proposition \ref{proposition 1}:} 
Take any $\varphi_0\in\mathds{R}^{mn}$. 
We first show that there exists a matrix $X\in\mathcal{D}$ such that $X_{kl} = W_{kl}(\varphi_{kl})$ and $\varphi-\varphi_0 \notin \text{im} (D^\top)$. 
   Now, since $n,m\geq 2$ by assumption, there exist nonzero vectors $v\in\mathds{R}^m$ and $w\in\mathds{R}^n$ such that $v^\top \mathds{1}_m=0$ and $w^\top \mathds{1}_n=0$.
   Using the definition of the incidence matrix \eqref{eqn: incidence matrix}, 
   \begin{align}
       (w^\top\otimes v^\top) D^\top &= (w^\top\otimes v^\top) \begin{bmatrix}
       I_n\otimes \mathds{1}_m &
       -\mathds{1}_n  \otimes I_m
   \end{bmatrix} \\
   &= \begin{bmatrix}
       w^\top \otimes v^\top \mathds{1}_m &
       -w^\top \mathds{1}_n  \otimes v^\top
   \end{bmatrix} =0.
   \end{align}
   Since $w\otimes v \neq 0$, there exists a $\varphi\in\mathds{R}^{mn}$ such that 
   \begin{equation}
       (w^\top\otimes v^\top) (\varphi-\varphi_0) \neq 0.
   \end{equation}
   Then, we know that $\varphi-\varphi_0\notin \text{im}(D^\top)$.
   Now, for this $\varphi$, let $X$ be defined by $X_{kl}=W_{kl}(\varphi_{kl})$. Then, by definition, $X\in\mathcal{D}$. Furthermore, by Lemma \ref{lemma 1}, $\varphi$ is the unique solution $\varphi'$ to $X_{kl}=W_{kl}(\varphi'_{kl})$. Hence, matrix $X\in\mathcal{D}$, but it is not contained in $\mathcal{R}(\varphi_0)$. Therefore, $\mathcal{R}(\varphi_0)\subsetneq \mathcal{D}$. \hfill \QED

So, Proposition \ref{proposition 1} shows that, when all switches are closed, there exist desired matrices $W_d\in\mathcal{D}$ that can not be reached. However, in the next section, it will be shown that the writing problem can be solved by appropriate use of the switches.

%% file: Writing/continuous.tex
\subsection{Controller}
In this section, we assume that the matrix $W(\varphi(0))$ is known, which is possible through reading as discussed in section III.
Now, consider a similar set-up of the crossbar array as for the reading case, where we add grounded voltage sources at the column terminals and ground the row terminals (i.e., $P_k^B=0$, for all $k\in\{1,\ldots,m\}$), but where the switches are not necessarily closed. 
In this case, the dynamics \eqref{eqn: reading dynamics} and \eqref{eqn: reading} are replaced by \vspace{-0.1cm}
\begin{equation}
\begin{aligned}
     \tfrac{d}{dt} \varphi_{kl} &= s_{kl} P_l^A, \quad -J_k^B = \sum_{j=1}^n W_{kj}(\varphi_{kj}) P^A_j s_{kj}. \label{eqn:writing dynamics}
\end{aligned}
\end{equation}
In addition, consider any desired matrix of memductances $W_d\in\mathcal{D}$, with $\mathcal{D}$ as in \eqref{eqn: desired memductance matrices set}, and take any desired $\epsilon>0$. 

Now, we will make full use of the switches and steer the memductance values of the memristors to the desired values one-by-one. In other words, select memristor $M_{kl}$ by letting the switches be such that  
\begin{equation} 
\begin{aligned} \label{eqn: controller switches}
     s_{ij} &= \left\{
        \begin{array}{ll}
            1 & \quad \text{if } i=k, j=l,   \\
            0 & \quad \text{otherwise}.
        \end{array}
    \right. 
    \end{aligned}
\end{equation}
The dynamics \eqref{eqn:writing dynamics} then leads to
\begin{equation} \label{eqn: controller flux}
\begin{aligned}
     \tfrac{d}{dt} \varphi_{ij} &= \left\{
        \begin{array}{ll}
            P_l^A & \quad \text{if } i=k, j=l,   \\
            0 & \quad \text{otherwise},
        \end{array}
    \right. 
    \end{aligned}
\end{equation}
and
\begin{equation} \label{eqn: controller current}
\begin{aligned}
     -J_i^B &= \left\{
        \begin{array}{ll}
            W_{il}(\varphi_{il})P_l^A & \quad \text{if } i=k,   \\
            0 & \quad \text{otherwise}.
        \end{array}
    \right. 
    \end{aligned}
\end{equation}
Hence, only the states, and thus the memductances, of the memristor $M_{kl}$ are changed and those of the other memristors remains unchanged. 

The writing problem for memristor $M_{kl}$ is now to design the input voltage signal at the $l$-th column terminal in such a way that $|W_{d,kl} - W_{kl}(\varphi_{kl}(\hat{T}_{kl}))|\leq\epsilon$ for some time $\hat{T}_{kl}\geq 0$. To this end, we claim that taking the input voltage signal according to the following algorithm solves the problem.
\begin{algorithm}\label{algorithm controller}
    Fix $T>0$, $\alpha>0$. If $| W_{d,kl} -W_{kl}(\varphi_{kl}(0))|> \epsilon$, then:
    \begin{enumerate} 
        \item Apply $P_l^A(t)=1$ to \eqref{eqn: controller flux} for $t\in[0,T]$.
        \item Set $i=1$.
        \item Measure $J_k^B(iT)$. 
        \item If $\left| W_{d,kl} + {J_k^B(iT)}/{P_l^A(iT)} \right|\leq \epsilon$, then stop.
        \item Apply $P_l^A(t)=\alpha \left( W_{d,kl} + {J_k^B(iT)}/{P_l^A(iT)} \right)$ to \eqref{eqn: controller flux} for $t\in(iT,(i+1)T]$.
        \item Set $i=i+1$ and go to step 3.
    \end{enumerate}
\end{algorithm}
The claim is formalized in the following theorem.
\begin{theorem}\label{theorem algorithm}
    Let $W_d\in\mathcal{D}$, $\epsilon>0$, $k\in\{1,\ldots, m\}$ and $l\in\{1,\dots,n\}$, and consider the switches in \eqref{eqn: controller switches}. Let $\alpha>0$ and $T>0$ be such that 
    \vspace{-0.5cm}
    
    \begin{equation}
        \alpha T < \tfrac{2}{\beta}, \label{eqn: alpha beta T condition}
    \end{equation}
    with $\beta>0$ as in Assumption \ref{assumption 2}.
    Then, for any  $\varphi_0\in\mathds{R}^{mn}$, the controller described in Algorithm \ref{algorithm controller} achieves
    \begin{equation} \label{eqn: epsilon bound}
        |W_{d,kl}-W_{kl}(\varphi_{kl}(\hat{T}_{kl}))|\leq\epsilon
    \end{equation} 
    for some $\hat{T}_{kl}\geq 0$, where $\varphi_{kl}(\cdot)$ is the solution to \eqref{eqn: controller flux}.    
\end{theorem}

This theorem thus tells us that the writing problem can be solved for any desired memductance matrix in $\mathcal{D}$, by applying the algorithm to all memristors $M_{kl}$, $k\in\{1,\ldots, m\}$ and $l\in\{1,\dots,n\}$.  

\quad \emph{Proof of Theorem \ref{theorem algorithm}:} 
Since $W_d\in\mathcal{D}$, we know that there exists a $\hat{\varphi}\in\mathds{R}^{mn}$ such that $W(\hat{\varphi})=W_d$ and that it is unique due to Lemma \ref{lemma 1}.
For simplicity of notation, we let $\varphi_d\in\mathds{R}$ denote the solution to $W_{kl}(\varphi_d) = W_{d,kl}$. 

To simplify notation, let $\varphi^i, P^i$ and $J^i$ respectively denote the flux $\varphi_{kl}(iT)$, voltage $P_l^A(iT)$, and current $J_k^B(iT)$ for $i=1,2,\ldots$.  
Combining the controller with the dynamics \eqref{eqn: controller flux}, we then find that  $ \varphi^i = \varphi^{i-1} + T P^i$, 
where
\begin{equation}\label{eqn: P^i}
    P^i = \left \{ \begin{array}{ll}
        1 & \text{if } i=1,  \\
        \alpha  \left( W_{kl}(\varphi_d) + \frac{J^{i-1}}{P^{i-1}} \right) & \text{if } i\geq 2.
    \end{array} \right.
\end{equation}
It then follows from \eqref{eqn: controller current} that \vspace{-0.1cm}
\begin{equation}\label{eqn: W_kl}
    W_{kl}(\varphi^i) = -\tfrac{J^i}{P^i}.
\end{equation} 
Equations \eqref{eqn: P^i} and \eqref{eqn: W_kl} tell us that \vspace{-0.1cm}
\begin{equation}
    P^i = \left \{ \begin{array}{ll}
        1 & \text{if } i=1,  \\
        \alpha  \left( W_{kl}(\varphi_d) -W_{kl}(\varphi^{i-1}) \right) & \text{if } i\geq 2.
    \end{array} \right.
\end{equation}
Clearly, the desired flux value $\varphi_d$ is an equilibrium, i.e.,  if $\varphi^{i-1}=\varphi_d$, then the algorithm will stop and no voltage will be applied, meaning that the flux $\varphi^j=\varphi(jT)$ will be equal to $\varphi_d$ for all $j\geq i$. 
 Now, consider candidate Lyapunov function $L(\varphi^i) = (\varphi^i-\varphi_d)^2$,
 which is non-negative and equal to zero only for $\varphi^i=\varphi_d$. 
Then, we have that
\begin{equation}
\begin{split}
\Delta L &\!:=\! L(\varphi^{i+1})-L(\varphi^i) \\
    &= \! (\varphi^i-\alpha T(W_{kl}(\varphi^i)-W_{kl}(\varphi_d))-\varphi_d)^2\! -\! (\varphi^i-\varphi_d)^2  \\
    &= -2\alpha T (W_{kl}(\varphi^i)-W_{kl}(\varphi_d))(\varphi^i-\varphi_d) \\
    &\quad + \alpha^2 T^2(W_{kl}(\varphi^i)-W_{kl}(\varphi_d))^2.
\end{split}
\end{equation}
\vspace{-0.5cm}

\noindent Now, using that $W_{kl}(\cdot)$ is $\beta-$Lipschitz continuous with constant $\beta>0$ and strictly monotone by Assumption \ref{assumption 2}, we observe that 
\begin{align}
   &-2\alpha T (W_{kl}(\varphi^i)-W_{kl}(\varphi_d))(\varphi^i-\varphi_d) \\
   &\qquad\leq - \tfrac{2\alpha T}{\beta} (W_{kl}(\varphi^i)-W_{kl}(\varphi_d))^2.
\end{align}
Using \eqref{eqn: alpha beta T condition}, this tells us that
\begin{align}
    \Delta L &\leq \left(\alpha^2T^2 -\tfrac{2\alpha T}{\beta}\right) (W_{kl}(\varphi^i)-W_{kl}(\varphi_d))^2 < 0,
\end{align}
for all $\varphi^i \neq \varphi_d$.

Hence, $\varphi_d$ is an asymptotically stable equilibrium, meaning that $\lim_{i\rightarrow\infty} \varphi^i = \varphi_d$. Therefore, the controller described in Algorithm \ref{algorithm controller} is such that \eqref{eqn: epsilon bound} holds  
    for some $\hat{T}_{kl}\geq 0$, which proves the theorem. 
\hfill \QED

With this, we found that the writing problem for the entire crossbar array, i.e., Problem \ref{writing problem}, can be solved by applying the algorithm to all memristors $M_{kl}$ successively, for $k\in\{1,\ldots,m\}$ and $l\in\{1,\ldots,n\}$. Namely, we find the following algorithm solves the problem.

\begin{algorithm}\label{algorithm 2}
  Let $\alpha>0$ and $T>0$ be such that \eqref{eqn: alpha beta T condition} holds for $\beta>0$ as in Assumption \ref{assumption 2}.  
    \begin{enumerate}  
    \item Let $k=1$, $l=1$.
        \item Set the switches as in \eqref{eqn: controller switches}.
        \item Apply Algorithm \ref{algorithm controller}. 
         Let $\hat{T}_{kl}\geq 0$ denote the time such that \eqref{eqn: epsilon bound} holds,
    with $\varphi_{kl}(\cdot)$ the solution to \eqref{eqn: controller flux}.   
        \item Set $t=0$. 
        \item If $k=m$ and $l=n$, then stop. 
        \item If $l\neq n$, let $l=l+1$. Otherwise, let $l=1$ and $k=k+1$. Go to step 2.
    \end{enumerate}
\end{algorithm}

Here, note that the finite time $\hat{T}_{kl}\geq 0$ as seen in step 3 of Algorithm \ref{algorithm 2} exists, for all $k\in \{1,\ldots,m\}$ and $l=\{1,\ldots,n\}$, by Theorem \ref{theorem algorithm}. Furthermore, note that, by choosing the switches appropriately, the algorithm steers each memristor independently to attain the desired memductance value by Theorem \ref{theorem algorithm}. 
As a result, we obtain the following: 

\begin{corollary}
    Algorithm \ref{algorithm 2} solves the writing problem, i.e., Problem \ref{writing problem} within finite time $\hat{T}_{11}+\ldots+\hat{T}_{mn}$.
\end{corollary}

\begin{remark}
    Note that the entire crossbar array can be written in $\max\{m,n\}$ steps. Namely, by applying Algorithm \ref{algorithm controller} to multiple memristors at the same time that are not in the same column and row, i.e., $M_{kl}$ and $M_{ij}$ can simultaneously be written when $k\neq i$ and $l\neq j$. This way, the memristors on the 'diagonals' can all be written at once.
\end{remark}

%% file: Applications/MatrixProduct.tex
\subsection{Matrix-vector products} 
An application of crossbar arrays is to compute matrix-vector products $c=Ab$, with $A\in\mathcal{D}\subset\mathds{R}^{m\times n}$ and $b\in\mathds{R}^n$, in one step.
Namely, consider a crossbar array with grounded voltage sources at the column terminals. Using the analysis in Sections III and IV, we can \textit{read} the memductance values of the memristors and then \textit{write} them in such a way that $A_{kl} = W_{d,kl}$ for all $k\in\{1,\ldots,m\}$ and  $l\in\{1,\ldots,n\}$.
Using what we saw in Section IV, we can then \textit{read} the product $c=Ab$ if we define the input voltage signals as $P_l^A(t) = f_l(t)$ for all $l\in\{1,\ldots,n\}$, where
\begin{equation} \label{eqn: application 1}
    f_l(t) = \left\{
        \begin{array}{ll}
            0 & \quad t\in [0, s-2\tau)   \\
            -b_l & \quad t \in [s-2\tau, s-\tau) \\ 
            b_l & \quad t \in [s-\tau, s+\tau) \\
            -b_l & \quad t \in [s+\tau, s+2\tau) \\
            0 & \quad t\in [s+2\tau,\infty). 
        \end{array}
    \right.
\end{equation} 
Here, $s\geq 2\tau$ and $b_l$ is the $l$-th element in the vector $b$. 
It then follows directly from \eqref{eqn: reading} and our analysis in Section III that the currents we measure at the end of the $m$ rows at time $s$ are precisely the entries of vector $c=Ab$, i.e., $c_k=-J_k^B(s)$, $k\in\{1,\ldots,m\}$.

\begin{remark}
    Note that the memductance values are always positive, i.e., $W_{kl}(\varphi_{kl}(t))>0$. 
    In practice, this is not restrictive as any matrix $A\in\mathds{R}^{m\times n}$ can be split as $A=B-C$, for some $B,C\in\mathds{R}_{>0}^{m\times n}$, which can be represented by two appropriately interconnected crossbar arrays, e.g.,
     \cite{Sun2019}.
\end{remark}

%% file: Applications/LeastSquares.tex
\subsection{Least-squares solutions} 
In this section, we take inspiration from \cite{Pates2019}, where a least-squares problem is solved using a resistive electrical circuit.
To this end, another application of crossbar arrays is to use them to compute the solution $x=-A^\dagger b$ to the least squares problem $\min_{x\in \mathds{R}^n} \Vert Ax+b \Vert_2$
for matrices $A\in\mathcal{D}\subset\mathds{R}^{m\times n}$ with full column rank. Hereto, consider a crossbar array with grounded current sources at the column terminals. Without loss of generality, assume that the memductance values of the memristors  are such that $A_{kl} = W_{kl}(\varphi_{0,kl})$ for all $k\in\{1,\ldots,m\}$ and  $l\in\{1,\ldots,n\}$.
If this is not the case, one can use the analysis in Section IV to steer the memductance values to the desired values of the matrix $A$. 

Now, the column terminals being grounded implies that the voltage potential is zero there, giving 
 \begin{align}
     P &=  \begin{bmatrix}  0 &  \cdots & 0 & P_1^B & \cdots & P_m^B\end{bmatrix}^{\top}.
 \end{align} 
The first $n$ lines of  \eqref{eqn:general crossbar array} then tell us that the input current given to each of the columns by the current sources is a linear combination of the voltage potentials measured at the end of the rows in the following way: 
\begin{equation}
    \begin{bmatrix} J_1^A \\  \vdots \\ J_n^A\end{bmatrix} 
    =  -\begin{bmatrix}
W_{11}(\varphi_{11})  & \cdots & W_{m1}(\varphi_{m1})  \\ 
\vdots & \ddots & \vdots \\
W_{1n}(\varphi_{1n}) & \cdots & W_{mn}(\varphi_{mn})
\end{bmatrix}\begin{bmatrix} P_1^B  \\  \vdots \\ P_m^B \end{bmatrix}. \label{eqn: reading current sources}
\end{equation}    

Similar to the previous subsection, let us choose the input current signals $J_l^A(t)=f_l(t)$ with $f_l(t)$ as in \eqref{eqn: application 1} 
for all $l\in\{1,\ldots,n\}$, where $b_l$ is the $l$-th element in the vector $b$.  It then follows directly from  \eqref{eqn: reading} that the voltage potentials we measure at the end of the $m$ rows at time $s$ are the entries of the vector $x=-A^\dagger b$.

%% file: Conclusion/conclusion.tex
We presented flux-controlled memristors and analyzed  memristive  crossbar arrays. In particular, we defined the problems of \textit{reading} and \textit{writing} the memductance values of the memristors in such circuits, i.e., the problems of determining and steering the memductance values. The reading problem was solved by choosing specific input voltage signals to the column terminals of the array, and by measuring the corresponding currents at the row terminals at certain times. The writing problem was solved by means of an algorithm in which input voltage signals to the column terminals of the array are updated by current measurements at the row terminals in such a way that the memductance values are steered towards desired values. The results were then applied to two applications, namely computing matrix-vector products and the solution to least-squares problems.

Future work will focus on expanding the set of applications, e.g., by allowing for matrix-vector multiplications with matrices that have negative elements.
In addition, future work will look into determining and steering the memductance values of charge-controlled memristors in crossbar arrays.